\documentclass[11pt,letterpaper,reqno]{amsart}

\usepackage{amsfonts}
\usepackage{graphicx}
\usepackage[normalem]{ulem}
\usepackage{xcolor}
\usepackage{bm} 
\usepackage{tikz}
\usepackage{amsmath, amssymb, amsthm}
\usepackage{latexsym}
\usepackage[utf8x]{inputenc}
\usepackage[T1]{fontenc}
\usepackage{fancyhdr}
\usepackage[colorlinks=true,urlcolor=blue]{hyperref}

\usepackage[margin=1in]{geometry}
\usepackage{todonotes}
\usepackage{mathtools}
\usetikzlibrary{arrows.meta,calc,angles,quotes,intersections}
\usepackage{comment}
\usepackage{environ}

\newif\ifshowexplanations
\showexplanationsfalse

\NewEnviron{explanation}{%
  \ifshowexplanations
    \par\smallskip
    {\color{blue}\BODY}%
    \par\smallskip
  \fi
}

\newcommand{\R}{\mathbb{R}}

\newcommand{\inner}[2]{\langle #1,#2\rangle}

\newcommand{\dd}{\,\mathrm{d}}

\newtheorem{theorem}{Theorem}[section]
\newtheorem{lemma}[theorem]{Lemma}
\newtheorem{proposition}[theorem]{Proposition}
\newtheorem{corollary}[theorem]{Corollary}

\newtheorem{remark}[theorem]{Remark}

\providecommand{\R}{\mathbb{R}}

\providecommand{\Sph}{\mathbb{S}}
\providecommand{\supp}{h} 
\providecommand{\width}{w} 
\providecommand{\Area}{\mathrm{Area}}

\title{Area Minimization Among Group-Invariant Planar Convex Bodies of Constant Width}

\author[J. Falc\'o, S. Myroshnychenko]{Javier Falc\'o, Sergii Myroshnychenko}

\date{\today}

\address{J. Falc\'o, An\`alisis Matem\'atico, Universidad de Valencia, Department of Mathematical Analysis, C/ Dr. Moliner, 50, 46100 Burjassot, Spain}
\email{francisco.j.falco@uv.es}

\address{S.~Myroshnychenko, Department of Mathematics and Statistics,  University of the Fraser Valley, Abbotsford, BC V2S 7M8,  Canada}
\email{serhii.myroshnychenko@ufv.ca}

\thanks{The first author was supported by grant PID2021-122126NB-C33 funded by MICIU/AEI/10.13039/501100011033 and by ERDF/EU. The second author was supported in part by NSERC}

\begin{explanation}
\tableofcontents
\end{explanation}

\subjclass[2020]{Primary 52A10, 52A40.}
\keywords{Convex body, constant width, group invariance.}

\begin{document}

\begin{abstract}
The classical Blaschke--Lebesgue theorem identifies the Reuleaux triangle as the planar convex body of constant width with minimum area. We investigate this extremal problem under prescribed symmetry constraints. Specifically, we classify the minimum-area convex bodies of constant width that are invariant under a finite group $G$ of isometries of $\mathbb{R}^2$ fixing the origin. For the exceptional reflection group $D_1$, the minimizers are precisely the Reuleaux triangles invariant under the prescribed reflection. If $G$ contains the half-turn $\mathcal R_\pi$, the disk is the unique minimizer. For odd $n\geq 3$, the minimizers are regular Reuleaux $n$-gons, unique up to rotation in the cyclic case $C_n$, and exactly those satisfying the prescribed reflection symmetry in the dihedral case $D_n$.
\end{abstract}

\maketitle

\section{Introduction}
Convex bodies of constant width constitute a classical yet remarkably rich class in convex geometry \cite{Martini2019}. Characterized by having the same width in every direction, they arise naturally in applications ranging from mechanical design \cite{ZCJ} to kinematics \cite{Reuleaux}, while displaying a striking diversity of geometric structures and extremal phenomena.

We work in the planar setting. Recall that a convex body \(K \subset \R^2\) is a compact convex set with non-empty interior. For a direction $u$ on the unit circle \(\Sph^1\) and the standard scalar product $\inner{\cdot}{\cdot}$, the support function of $K$ in the direction $u$ is
$
\supp(K,u)=\max_{x \in K}\inner{x}{u},
$
and the width of \(K\) in direction \(u\) is
$
\width(K,u)=\supp(K,u)+\supp(K,-u).
$
We say \(K\) has constant width \(w>0\) if \(\width(K,u)\equiv w\) for every \(u\in\Sph^1\). The most familiar examples are the disk $B_w$ of diameter \(w\) and the Reuleaux $n$-gons $R_n(w)$ of width~\(w\).

The classical Blaschke--Lebesgue theorem \cite{Blaschke1915,Firey1960, Hynd2024,Lebesgue1914} characterizes the extremal area properties of convex bodies of constant width $w$. It states that the Reuleaux triangle uniquely minimizes area. On the opposite extreme, it is well known that as a consequence of the isoperimetric inequality, the disk $B_w$ of diameter $w$ uniquely maximizes area. Specifically, for any convex body $K$ of constant width $w$,
\[
\Area(R_3(w)) \leq \Area(K) \leq \Area(B_w).
\]
Indeed, the upper bound follows from Barbier's theorem \cite{Ba} (which asserts that every planar convex body of constant width $w$ has perimeter $\pi w$) together with the classical isoperimetric inequality, \cite[Section 7.2]{Schneider2014}.

In this work, we investigate how the area-minimization property is affected when symmetry constraints are imposed. Specifically, we consider convex bodies of constant width that are invariant under the action of a finite group $G$ of isometries of $\R^2$ fixing the origin. Note that by translating any finite group of Euclidean isometries to one of its common fixed points, the assumption that the group fixes the origin entails no loss of generality. Such groups define a natural action on subsets of $\R^2$ by
\[
G \times \mathcal{P}(\R^2) \to \mathcal{P}(\R^2), \qquad (g, K) \mapsto g \cdot K := \{g(x) : x \in K\},
\]
where $\mathcal{P}(\R^2)$ denotes the power set of $\R^2$. A set $K \subset \R^2$ is $G$-invariant if $g \cdot K = K$ for all $g \in G$.

Finite subgroups of the group of isometries of \(\R^2\) fixing the origin are precisely the cyclic and dihedral groups. It is a classical result (see, e.g., \cite[Theorem 19.1]{Armstrong1988}) that any such non-trivial group \(G\) must be either:
\begin{itemize}
    \item A cyclic group \(C_n\), generated by a rotation \(\mathcal R_{2\pi/n}\) of angle \(2\pi/n\), with \(n\ge 2\), $n \in \mathbb{N}$, or
    \item A dihedral group \(D_n\) with \(n\ge 1\), $n \in \mathbb{N}$, generated by \(\mathcal R_{2\pi/n}\) together with a reflection \(S\) across a line through the origin (which we take to be the horizontal line \(\{y = 0\}\) for simplicity). For \(n=1\), this means simply \(D_1=\{I,S\}\), the group generated by the identity $I$ and a reflection $S$.
\end{itemize}
For \(n\ge 2\), these groups arise as the symmetry group of a regular \(n\)-gon centered at the origin. The cyclic group \(C_n\) corresponds to the rotational symmetries of such a polygon, while the dihedral group \(D_n\) includes both rotations and reflections. The exceptional case \(D_1\) is the reflection-only group.
Our main result provides a complete classification.

\begin{theorem}\label{thm:main}
Let \( G \) be a non-trivial finite subgroup of isometries of \( \R^2 \) fixing the origin, and let \( w > 0 \). Then the minimizers among the \( G \)-invariant convex bodies of constant width \( w \) are given as follows:
\begin{itemize}
    \item If \( G = D_1 \), then the minimizers are precisely the regular Reuleaux triangles $R_3(w)$ that are invariant under the prescribed reflection.
    \item If \( G \) contains the rotation \( \mathcal R_\pi \) by angle \( \pi \) (equivalently, if \( G = C_n \) or \( G = D_n \) with \( n \) even), then the unique minimizer is the disk \( B_w \) of diameter \( w \).
    \item If \( G = C_n \) with odd \( n\ge 3 \), then the minimizers are precisely the centered regular Reuleaux \( n \)-gons of width \( w \); equivalently, the minimizer is unique up to rotation.
    \item If \(G=D_n\) with odd \(n\ge3\), then the minimizers are exactly the two centered regular Reuleaux \(n\)-gons of width \(w\) invariant under the prescribed dihedral action. Equivalently, they are the two orientations for which the prescribed reflection axis is a symmetry axis of the Reuleaux polygon.
\end{itemize}
\end{theorem}

The case of $C_n$ and $D_n$ for even $n$ is a well-known consequence of the classical Blaschke-Lebesgue theorem: a half-turn forces the body to be centrally symmetric, and the only centrally symmetric convex body of constant width is the disk.
The case \(G=D_1\) is an immediate consequence of the classical Blaschke-Lebesgue theorem and is treated separately. 

For the remaining cases, we briefly outline the argument. No smoothness
assumption on $K$ or its support function is required. Although the
classical formulas are first derived under the temporary assumption
$h\in C^2$, they are subsequently extended to arbitrary convex bodies
by a distributional argument and Fejér approximation. Indeed, for a
general support function, convexity means that
\(
\mu:=h+h''
\)
is a nonnegative measure on $\mathbb S^1$. The constant-width identity
$h(\theta)+h(\theta+\pi)=w$ forces $\mu$ to be absolutely continuous:
\[
\mu=\rho(\theta)\,\dd\theta,
\qquad
0\leq\rho\leq w,
\qquad
\rho(\theta)+\rho(\theta+\pi)=w
\quad\text{a.e.},
\]
while $G$-symmetry imposes the corresponding periodicity conditions on
$\rho$. Using Fejér means, we then extend the support-function area
formula
\(
\operatorname{Area}(K)
=
\frac12\int_0^{2\pi}
\bigl(h^2-(h')^2\bigr)\,\dd\theta
\)
and its Fourier representation to these nonsmooth profiles. This
reduces area minimization to the maximization of an explicit quadratic
functional of $\rho$. Finally, a symmetric rearrangement argument shows
that, for odd rotational symmetries, every maximizing profile takes only
the values $0$ and $w$ almost everywhere, and the corresponding body is $R_n(w)$. For more information on the methods of harmonic analysis and distributions, we refer the reader to \cite{Edwards1}, \cite{Edwards2} and references therein.

\section{Notation and Preliminaries}

Recall that the Reuleaux triangle $R_3(w)$ is constructed as follows: begin with an equilateral triangle of side length $w$, then replace each side with a circular arc centered at the opposite vertex and passing through the endpoints of the side. The resulting shape is a convex body of constant width $w$ that is not a disk (see Figure \ref{fig:reuleaux_triangle}). The generalized Reuleaux polygon $R_n(w)$ for odd $n \geq 3$, is constructed analogously by replacing the triangle with the corresponding regular $n$-gon whose vertices lie on a circle of radius $\frac{w}{2 \cos \left( \frac{\pi}{2n}\right)}$. Then every boundary arc has radius $w$, joins consecutive vertices, and is centered at the appropriate opposite vertex.

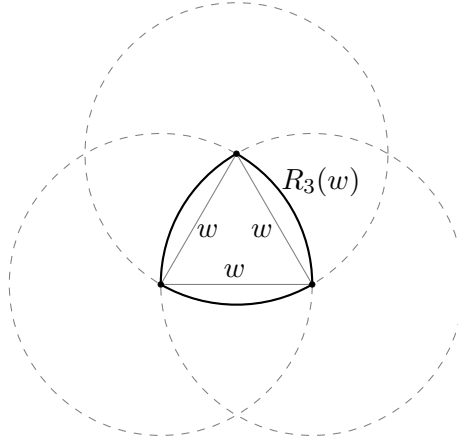
\begin{figure}[ht]
\centering
\begin{tikzpicture}[scale=2]
  \coordinate (A) at (0,0);
  \coordinate (B) at (1,0);
  \coordinate (C) at (0.5,{sqrt(3)/2});

  \draw[thin,gray,dashed] (A) circle (1);
  \draw[thin,gray,dashed] (B) circle (1);
  \draw[thin,gray,dashed] (C) circle (1);

  \draw[thick] (B) arc[start angle=0,end angle=60,radius=1];
  \draw[thick] (C) arc[start angle=120,end angle=180,radius=1];
  \draw[thick] (A) arc[start angle=240,end angle=300,radius=1];

   \draw[gray] (A) -- (B) -- (C) -- cycle;

  \filldraw (A) circle (0.5pt);
  \filldraw (B) circle (0.5pt);
  \filldraw (C) circle (0.5pt);
  
  \node[above right] at (0.17,0.25) {$w$};
  \node[above right] at (0.53,0.25) {$w$};
  \node[above right] at (0.35,-0.02) {$w$};
  \node[above right] at (0.73,0.52) {$R_3(w)$};
\end{tikzpicture}
\caption{Construction of the Reuleaux triangle $R_3(w)$.}
\label{fig:reuleaux_triangle}
\end{figure}

When we study the classes $G=C_n$ and $G=D_n$, we realize the group as acting about the origin (and, for $D_n$, with reflection axes through the $x$-axis). Thus, the symmetry center is fixed from the outset. 
Throughout we identify unit directions in the plane with angles $\theta\in[0,2\pi)$ and write
\[
u(\theta)=(\cos\theta,\sin\theta),\qquad u^\perp(\theta)=(-\sin\theta,\cos\theta).
\]
We denote by $K\subset\R^2$ a planar convex body (compact, convex set with non-empty interior), by $h$ its support function, and by $w>0$ the constant width.

We recall that the support function $h$ is defined by $h(\theta)=\max_{x\in K}\inner{x}{u(\theta)}$.
Note that the support function $h$ is defined with respect to the origin, which is fixed by the group action. In particular, 
$h(\theta)$ is the signed distance from the origin to the supporting line of $K$ whose outward normal makes angle $\theta$ with the $x$-axis (\cite{Schneider2014}, Section 2.4).
Since $K$ has constant width $w$, we have
$
h(\theta)+h(\theta+\pi)=w
$
for all $\theta \in [0, 2\pi]$.
Also, $h$ is a $2\pi$-periodic function, and we have
\[
K = \bigcap_{\theta \in [0,2\pi)} \{ x \in \mathbb{R}^2 : \langle x, u(\theta) \rangle \le h(\theta) \}.
\]

Recall (e.g., \cite{Ga}) that if we translate $K$ by a vector $t\in\R^2$, then 
\(
h_{K+t}(\theta)=h_K(\theta)+\inner{t}{u(\theta)}
\)
and if we rotate $K$ by angle $\alpha$ about the origin, then
\(
h_{\mathcal R_\alpha K}(\theta)=h_K(\theta-\alpha).
\)

\begin{explanation}
    \begin{proof}
For any $x\in K$, we have $x+t\in K+t$, so
\[
h_{K+t}(\theta)=\max_{y\in K+t}\inner{y}{u(\theta)}=\max_{x\in K}\inner{x+t}{u(\theta)}=\max_{x\in K}\inner{x}{u(\theta)}+\inner{t}{u(\theta)}=h_K(\theta)+\inner{t}{u(\theta)}.
\]
Let $R_\alpha$ be the rotation matrix by angle $\alpha$. Then
\[
h_{R_\alpha K}(\theta)=\max_{y\in R_\alpha K}\inner{y}{u(\theta)}=\max_{x\in K}\inner{R_\alpha x}{u(\theta)}=\max_{x\in K}\inner{x}{R_\alpha^{-1}u(\theta)}=\max_{x\in K}\inner{x}{u(\theta-\alpha)}=h_K(\theta-\alpha).
\]
\end{proof}

For completeness, we recall the construction of the generalized Reuleaux polygon. For any odd integer \(n \geq 3\), let \(P_n\) be a regular \(n\)-gon with vertices \(V_0,\dots,V_{n-1}\) inscribed in a circle centered at the origin. The \emph{regular Reuleaux \(n\)-gon} \(R_n(w)\) of width \(w\) is obtained by replacing each side \([V_j,V_{j+1}]\) of \(P_n\) with the circular arc of radius \(w\) centered at \(V_{j+(n+1)/2}\), where indices are taken modulo \(n\). This is the unique vertex of \(P_n\) equidistant from the endpoints of the side \([V_j,V_{j+1}]\). Here, $w$ equals the diameter of $P_n$. The resulting shape is a convex body of constant width \(w\) that is invariant under the action of both \(C_n\) and \(D_n\). Note that this construction is only valid for odd values of \(n\).
\end{explanation}

\begin{remark}
    First, we derive the geometric identities under the $C^2$ assumption on $h$. Later, after reformulating the problem in terms of the curvature profile $\rho$, we remove this regularity assumption by approximating admissible profiles with their Fej\'er means, \cite{Edwards1}. Since the Fej\'er operator is convolution on the circle with a translation-invariant kernel, it preserves the relations $0\leq \rho\leq w$, $\rho(\theta)+\rho(\theta+\pi)=w$, and the symmetry constraints encoded by angular shifts.
\end{remark}

\subsection{Case $G=D_1$}
We first dispose of the exceptional reflection-only case.

\begin{proposition}[The reflection-only case]\label{prop:D1}
Let $G=D_1=\{I,S\}$, where $S$ is reflection across a prescribed line through the origin. Among the $G$-invariant convex bodies of constant width $w$, the minimizers are precisely the Reuleaux triangles of width $w$ that are invariant under $S$, equivalently, those having the prescribed reflection line as one of their symmetry axes.
\end{proposition}
\begin{proof}
By the Blaschke–Lebesgue theorem, every admissible body has area at least that of the Reuleaux triangle, with equality only for a Reuleaux triangle. Such a triangle is invariant under the prescribed reflection precisely when one of its symmetry axes coincides with the prescribed reflection axis. Conversely, every such Reuleaux triangle is admissible and attains equality.
\end{proof}


\subsection{Preliminary results}

From now on, we assume that $G$ contains a non-trivial rotation. Thus, $G=C_n$ or $G=D_n$ with $n\ge2$.
\begin{explanation}
Let us begin by deriving the formula for the boundary $\partial K$ in terms of the support function $h$. This is a classical result in convex geometry, but we include the proof for completeness and to set up notation for later sections.
\end{explanation}
 The contact point $\gamma(\theta)$ on the boundary $\partial K$ where the supporting line with normal $u(\theta)$ supports $K$ is the unique point such that $\inner{\gamma(\theta)}{u(\theta)}=h(\theta)$, see Figure \ref{fig:support}. The following theorem gives an explicit formula for $\gamma(\theta)$ in terms of $h$ and its derivative.

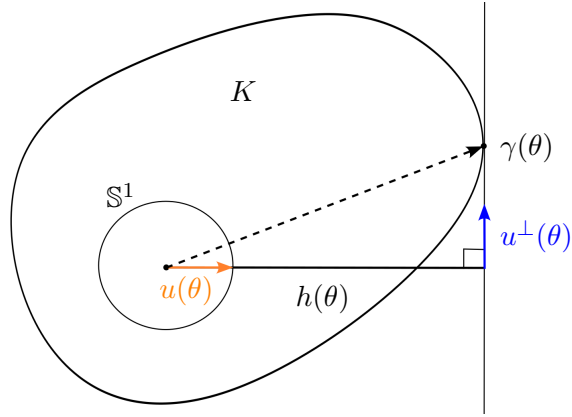
\begin{figure}[h]
    \centering
       \begin{tikzpicture}
\node[anchor=south west, inner sep=0] (img) at (0,0){\includegraphics[scale=0.2]{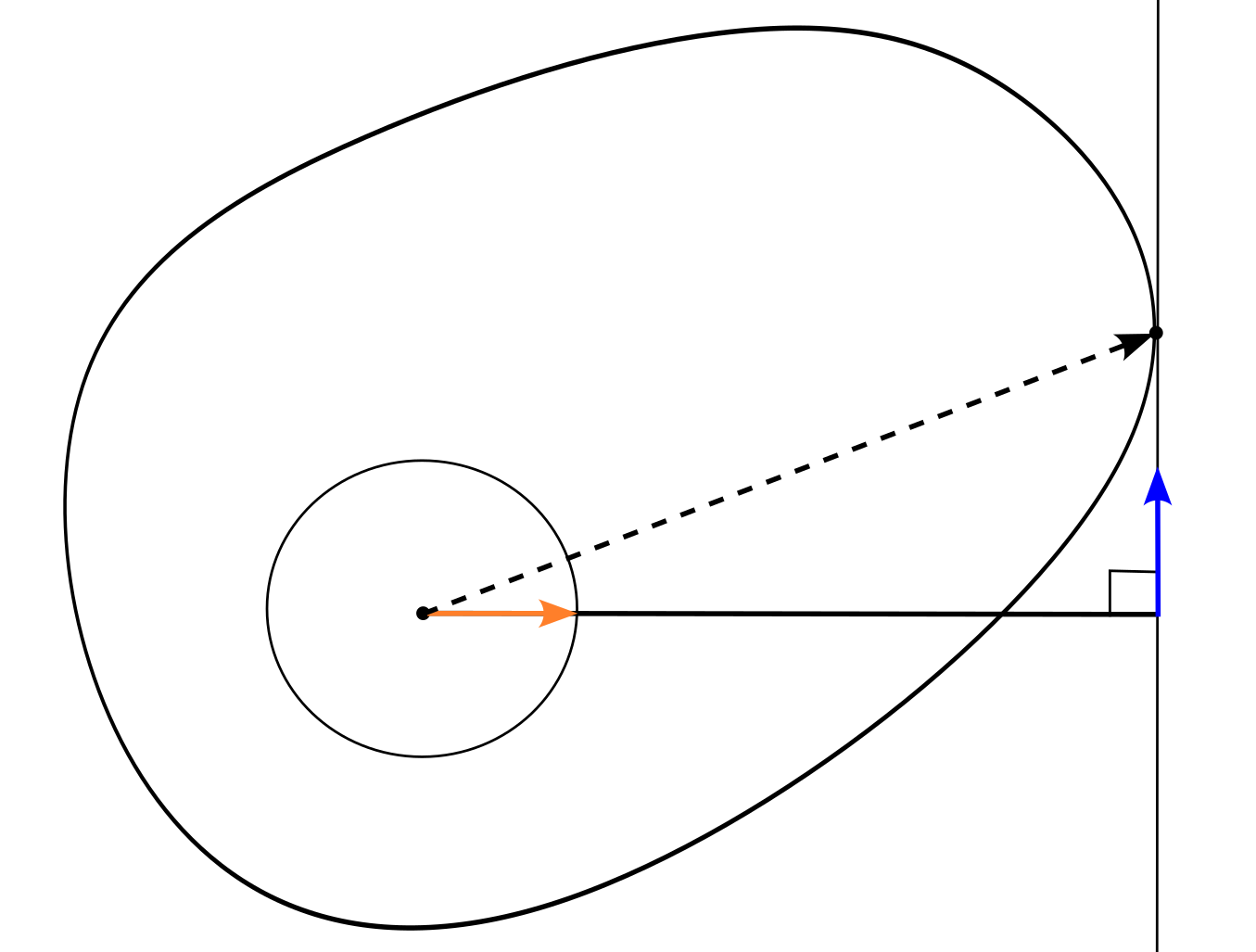}};
\node[above right] at (2.2,1.3) {\textcolor{orange}{$u(\theta)$}};
\node[above right] at (4,1.2) {$h(\theta)$};
\node[above right] at (6.7,2) {$\textcolor{blue}{u^\perp(\theta)}$};
\node[above right] at (6.7,3.2) {$\gamma(\theta)$};
\node[above right] at (1.5,2.6) {$\mathbb{S}^1$};
\node[above right] at (3.13,4) {$K$};
\end{tikzpicture}
\caption{Support function $h(\theta)$ and contact point $\gamma(\theta)$.}
\label{fig:support}
\end{figure}

\begin{theorem}[\cite{Martini2019} 
]
\label{thm:contact}
For $h \in C^2$, the boundary $\partial K$ is parametrized by a $2\pi$-periodic curve $\gamma:[0,2\pi)\to\R^2$, where the point $\gamma(\theta)$ is the unique point on $\partial K$ where the supporting line with normal $u(\theta)$ touches $K$.
This curve is explicitly given by $\gamma(\theta) = h(\theta) u(\theta) + h'(\theta) u^\perp (\theta)$.
\end{theorem}

\begin{explanation}
\begin{proof}
To derive this formula, define
\[
u(\theta) = (\cos\theta,\sin\theta), \quad
u^\perp(\theta) = (-\sin\theta,\cos\theta).
\]

Note that $u^\perp(\theta)$ is perpendicular to $u(\theta)$ and points in the direction of increasing $\theta$ (counterclockwise).

The point $\gamma(\theta)$ on $\partial K$ where the supporting line with normal $u(\theta)$ touches $K$ must lie in the affine line
\[
\{ x : \langle x,u(\theta)\rangle = h(\theta)\}.
\]

Therefore, since $\{u(\theta),u^\perp(\theta)\}$ is an orthonormal basis of $\R^2$, we can write
\[
\gamma(\theta)=\lambda(\theta) u(\theta)+\mu(\theta) u^\perp(\theta)
\]
for some functions $\lambda,\mu$.

The constraint $\inner{\gamma(\theta)}{u(\theta)}=h(\theta)$ gives
\[
\lambda(\theta)\inner{u(\theta)}{u(\theta)}+\mu(\theta)\inner{u^\perp(\theta)}{u(\theta)}=h(\theta).
\]
Since $\inner{u(\theta)}{u(\theta)}=1$ and $\inner{u^\perp(\theta)}{u(\theta)}=0$, we get
\[
\lambda(\theta)=h(\theta).
\]

It remains to identify $\mu(\theta)$. Fix $\theta$ and consider the scalar function
\[
g_\theta(\varphi)\coloneqq \inner{\gamma(\theta)}{u(\varphi)}-h(\varphi).
\]
By definition of the support function,
\[
h(\varphi)=\max_{x\in K}\inner{x}{u(\varphi)}\ge \inner{\gamma(\theta)}{u(\varphi)}
\qquad\text{for all }\varphi,
\]
so $g_\theta(\varphi)\le 0$ for all $\varphi$. At $\varphi=\theta$ we have equality, because $\gamma(\theta)$ lies on the supporting line with normal $u(\theta)$:
\[
g_\theta(\theta)=\inner{\gamma(\theta)}{u(\theta)}-h(\theta)=0.
\]
Hence $\varphi=\theta$ is a point where the differentiable function $g_\theta$ attains its maximum, so
\[
g_\theta'(\theta)=0.
\]
Differentiating gives
\[
g_\theta'(\varphi)=\inner{\gamma(\theta)}{u'(\varphi)}-h'(\varphi)
=\inner{\gamma(\theta)}{u^\perp(\varphi)}-h'(\varphi).
\]
Evaluating at $\varphi=\theta$ and using
\[
\gamma(\theta)=h(\theta)u(\theta)+\mu(\theta)u^\perp(\theta),
\]
we obtain
\[
0=g_\theta'(\theta)=\inner{\gamma(\theta)}{u^\perp(\theta)}-h'(\theta)=\mu(\theta)-h'(\theta).
\]
Therefore
\[
\mu(\theta)=h'(\theta).
\]
Substituting the formulas for $\lambda(\theta)$ and $\mu(\theta)$ yields
\[
\gamma(\theta)=h(\theta)u(\theta)+h'(\theta)u^\perp(\theta),
\]
that is,
\[
\gamma(\theta)
=
\begin{pmatrix}
h(\theta)\cos\theta-h'(\theta)\sin\theta\\
h(\theta)\sin\theta+h'(\theta)\cos\theta
\end{pmatrix}.
\]

\end{proof}
\end{explanation}

In these settings, the formula for the tangent vector can be computed directly.

\begin{proposition}
\label{definitiongammaprime}
The tangent vector to the boundary is
\(
\gamma'(\theta) = (h(\theta)+h''(\theta))\,u^\perp(\theta).
\)
\end{proposition}
\begin{explanation}
\begin{proof}
From the formula $\gamma(\theta)=h(\theta)u(\theta)+h'(\theta)u^\perp(\theta)$, we differentiate:
\begin{align*}
\gamma'(\theta)&=h'(\theta)u(\theta)+h(\theta)u'(\theta)+h''(\theta)u^\perp(\theta)+h'(\theta)(u^\perp)'(\theta)\\
&=h'(\theta)u(\theta)+h(\theta)u^\perp(\theta)+h''(\theta)u^\perp(\theta)+h'(\theta)(-u(\theta))\\
&=(h'(\theta)-h'(\theta))u(\theta)+(h(\theta)+h''(\theta))u^\perp(\theta)\\
&=(h(\theta)+h''(\theta))u^\perp(\theta).
\end{align*}
\end{proof}
\end{explanation}
We also mention the well-known convexity condition. 
\begin{explanation}
As we move along the boundary parametrized by $\theta$, the tangent vector is $\gamma'(\theta)=(h+h'')u^\perp(\theta)$.

For the boundary to be traced consistently in one direction (counterclockwise), we need $h+h''\ge 0$.
If $h+h''$ were negative, the curve would backtrack, which is incompatible with the boundary of a convex set. Again, this is a classical result in convex geometry, but we include the formal proof of this equivalence for completeness.
\end{explanation}

\begin{lemma}[\cite{Schneider2014}]
\label{lemmaConvexity}
A $2\pi$-periodic function $h \in C^2( \mathbb S^1)$ is the support function of a compact convex set  if and only if $h+h'' \geq 0$ on $\mathbb{S}^1$.

\end{lemma}

\begin{explanation}
\begin{proof}
Set
\[
\rho(\theta)\coloneqq h(\theta)+h''(\theta).
\]
By the previous proposition,
\[
\gamma'(\theta)=\rho(\theta)u^\perp(\theta).
\]

Assume first that $K$ is convex. Then as $\theta$ increases, the outer normal to the supporting line is $u(\theta)$, so the positively oriented tangent direction to $\partial K$ is $u^\perp(\theta)$. Hence the velocity vector $\gamma'(\theta)$ must be a nonnegative multiple of $u^\perp(\theta)$ for every $\theta$. Since $\gamma'(\theta)=\rho(\theta)u^\perp(\theta)$, we conclude that $\rho(\theta)\ge 0$ for all $\theta$, namely
\[
h(\theta)+h''(\theta)\ge 0 \quad \text{for all }\theta.
\]

Conversely, assume that $\rho(\theta)\ge 0$ for all $\theta$. Fix $\alpha\in[0,2\pi)$ and consider the scalar function
\[
p_\alpha(\theta)\coloneqq \inner{\gamma(\theta)}{u(\alpha)}.
\]
Differentiating and using the formula for $\gamma'$ gives
\[
p_\alpha'(\theta)=\inner{\gamma'(\theta)}{u(\alpha)}=\rho(\theta)\inner{u^\perp(\theta)}{u(\alpha)}=\rho(\theta)\sin(\alpha-\theta).
\]
Therefore $p_\alpha'(\theta)\ge 0$ on $[\alpha-\pi,\alpha]$ and $p_\alpha'(\theta)\le 0$ on $[\alpha,\alpha+\pi]$. It follows that $p_\alpha$ attains its global maximum at $\theta=\alpha$. Since $\gamma(\alpha)$ lies on the supporting line with normal $u(\alpha)$,
\[
p_\alpha(\alpha)=\inner{\gamma(\alpha)}{u(\alpha)}=h(\alpha).
\]
Hence for every $\theta$,
\[
\inner{\gamma(\theta)}{u(\alpha)}=p_\alpha(\theta)\le p_\alpha(\alpha)=h(\alpha).
\]
Because $\alpha$ was arbitrary, the image of $\gamma$ is contained in
\[
\bigcap_{\alpha\in[0,2\pi)}\{x\in\R^2:\inner{x}{u(\alpha)}\le h(\alpha)\}=K.
\]
Moreover, equality holds at $\theta=\alpha$, so for every $\alpha$ the line
\[
\{x\in\R^2:\inner{x}{u(\alpha)}=h(\alpha)\}
\]
supports the image of $\gamma$. Thus the image of $\gamma$ is exactly the boundary of the intersection above. Since that intersection is an intersection of closed half-planes, it is convex. Therefore $K$ is convex.
\end{proof}
\end{explanation}

When $\partial K$ is smooth and encloses a strictly convex set, let
\begin{equation}\label{rho}
    \rho(\theta)\coloneqq h(\theta)+h''(\theta),
\end{equation}

which is the (pointwise) radius of curvature of $\partial K$ at the point with the outer normal direction $u(\theta)$. See \cite[Theorem 5.1.1(d)--(e)]{Martini2019} or \cite[Corollary 2.5.3, formula (2.60)]{Schneider2014}. We observe that for constant width, the curvature radius satisfies a simple ``antipodal'' constraint.
\begin{explanation}
In general, $\rho$ can be interpreted as a generalized curvature radius, which may vanish at some points (corresponding to corners) but cannot be negative.
\end{explanation}

\begin{remark}[Constant width in terms of $\rho$]
If $h(\theta)+h(\theta+\pi)=w$ for all $\theta$, then differentiating twice gives $h''(\theta)+h''(\theta+\pi)=0$. Adding yields
\(
\rho(\theta)+\rho(\theta+\pi)=w.
\)

\end{remark}

It is convenient to express the area of $K$ in terms of $h$ and $\rho$.

\begin{explanation}
    Green's theorem relates a double integral over a region to a line integral around its boundary. It's a two-dimensional version of the fundamental theorem of calculus.

\begin{theorem}[Green's Theorem]
Let $D\subset\R^2$ be a region whose boundary $\partial D$ is a simple closed curve oriented counterclockwise. Let $P(x,y)$ and $Q(x,y)$ be continuously differentiable functions on $D$. Then
\[
\iint_D\left(\frac{\partial Q}{\partial x}-\frac{\partial P}{\partial y}\right)\dd x\dd y=\oint_{\partial D}(P\dd x+Q\dd y).
\]
\end{theorem}

\end{explanation}

\begin{theorem}[\cite{Bonesen1987}, Section 39]
  \label{thm:area-formula}
Assume that the support function $h$ of $K$ belongs to $C^2$. Then
the area of $K$ is given by
\[
\operatorname{Area}(K)=\frac12\int_0^{2\pi} h(\theta)\,\rho(\theta)\dd\theta=\frac{1}{2}\int_0^{2\pi}\bigl(h^2+hh''\bigr)\dd\theta=\frac{1}{2}\int_0^{2\pi}\bigl(h^2-(h')^2\bigr)\dd\theta.
\]
\end{theorem}

\begin{explanation}
\begin{proof}

    To compute the area of our region $K$ we use Green's theorem with the standard choice
$P(x,y)=-\tfrac{y}{2},\quad Q(x,y)=\tfrac{x}{2},$
so that
\[
\frac{\partial Q}{\partial x}-\frac{\partial P}{\partial y}=1,
\]
and therefore
\[
\text{Area}(K)=\frac{1}{2}\oint_{\partial K}(x\dd y-y\dd x).
\]

So we need to compute $x\dd y-y\dd x$ in terms of $\theta$.

    Since the boundary of $K$ is parametrized by
\[
\gamma(\theta)=(x(\theta),y(\theta))=\begin{pmatrix}h(\theta)\cos\theta-h'(\theta)\sin\theta\\h(\theta)\sin\theta+h'(\theta)\cos\theta\end{pmatrix},
\]
we have
\[
x(\theta)=h(\theta)\cos\theta-h'(\theta)\sin\theta,
\]
so
\begin{align*}
\frac{\dd x}{\dd\theta}&=h'(\theta)\cos\theta+h(\theta)(-\sin\theta)-h''(\theta)\sin\theta-h'(\theta)\cos\theta\\
&=-h(\theta)\sin\theta-h''(\theta)\sin\theta\\
&=-(h(\theta)+h''(\theta))\sin\theta,
\end{align*}
and
\[
y(\theta)=h(\theta)\sin\theta+h'(\theta)\cos\theta,
\]
so
\begin{align*}
\frac{\dd y}{\dd\theta}&=h'(\theta)\sin\theta+h(\theta)\cos\theta+h''(\theta)\cos\theta-h'(\theta)\sin\theta\\
&=h(\theta)\cos\theta+h''(\theta)\cos\theta\\
&=(h(\theta)+h''(\theta))\cos\theta.
\end{align*}

Therefore,
\[
\left\{\begin{aligned}
\dd x&=-(h+h'')\sin\theta\,\dd\theta\\
 \dd y&=(h+h'')\cos\theta\,\dd\theta.
\end{aligned}\right.
\]

Thus
\begin{align*}
x\dd y-y\dd x&=\bigl[(h\cos\theta-h'\sin\theta)(h+h'')\cos\theta+(h\sin\theta+h'\cos\theta)(h+h'')\sin\theta\bigr]\dd\theta\\
&=(h+h'')\bigl[(h\cos\theta-h'\sin\theta)\cos\theta+(h\sin\theta+h'\cos\theta)\sin\theta\bigr]\dd\theta\\
&=(h+h'')\bigl[h\cos^2\theta-h'\sin\theta\cos\theta+h\sin^2\theta+h'\cos\theta\sin\theta\bigr]\dd\theta\\
&=(h+h'')\bigl[h(\cos^2\theta+\sin^2\theta)+h'(\cos\theta\sin\theta-\sin\theta\cos\theta)\bigr]\dd\theta\\
&=(h+h'')h\,\dd\theta.
\end{align*}

Finally, we integrate to get the area.

\[
\text{Area}(K)=\frac{1}{2}\oint_{\partial K}(x\dd y-y\dd x)=\frac{1}{2}\int_0^{2\pi}h(\theta)(h(\theta)+h''(\theta))\dd\theta=\frac{1}{2}\int_0^{2\pi}(h^2+hh'')\dd\theta.
\]

By definition of $\rho$, we have $h+h''=\rho$ we have that 
\[\text{Area}(K)=\frac{1}{2}\int_0^{2\pi}h(\theta)(h(\theta)+h''(\theta))\dd\theta=\frac{1}{2}\int_0^{2\pi}h(\theta)\rho(\theta)\dd\theta,\]
so it only remains to obtain the last equality in the statement of the theorem. For this, note that, using Integration by Parts, we have 
\[
\int_0^{2\pi}hh''\dd\theta=\left[hh'\right]_0^{2\pi}-\int_0^{2\pi}(h')^2\dd\theta.
\]

Since $h$ is $2\pi$-periodic, $h(2\pi)=h(0)$ and $h'(2\pi)=h'(0)$, so $[hh']_0^{2\pi}=0$.

Therefore,
\[
\int_0^{2\pi}hh''\dd\theta=-\int_0^{2\pi}(h')^2\dd\theta.
\]

Substituting back:
\[
\text{Area}(K)=\frac{1}{2}\int_0^{2\pi}(h^2+hh'')\dd\theta=\frac{1}{2}\int_0^{2\pi}(h^2-(h')^2)\dd\theta.
\]
\end{proof}
\end{explanation}

\begin{explanation}
    
\textbf{Brief Review of Fourier Series}

A $2\pi$-periodic function $f(\theta)$ can be expanded as
\[
f(\theta)=a_0+\sum_{k=1}^\infty(a_k\cos k\theta+b_k\sin k\theta),
\]
where the Fourier coefficients are given by
\[
a_0=\frac{1}{2\pi}\int_0^{2\pi}f(\theta)\dd\theta,\quad a_k=\frac{1}{\pi}\int_0^{2\pi}f(\theta)\cos k\theta\dd\theta,\quad b_k=\frac{1}{\pi}\int_0^{2\pi}f(\theta)\sin k\theta\dd\theta.
\]

\textbf{Orthogonality relations:}
\begin{align*}
\int_0^{2\pi}\cos k\theta\cos\ell\theta\dd\theta&=\begin{cases}0,&k\ne\ell,\\\pi,&k=\ell\ge 1,\\2\pi,&k=\ell=0,\end{cases}\\
\int_0^{2\pi}\sin k\theta\sin\ell\theta\dd\theta&=\begin{cases}0,&k\ne\ell,\\\pi,&k=\ell\ge 1,\end{cases}\\
\int_0^{2\pi}\cos k\theta\sin\ell\theta\dd\theta&=0\quad\text{for all }k,\ell.
\end{align*}

\textbf{Parseval's identity:}
\[
\int_0^{2\pi}f(\theta)^2\dd\theta=2\pi a_0^2+\pi\sum_{k=1}^\infty(a_k^2+b_k^2).
\]

\end{explanation}

For now, assume that $h\in C^2$ and write its Fourier series as
\[
h(\theta) = a_0 + \sum_{k=1}^\infty (a_k\cos k\theta + b_k\sin k\theta).
\]
By the standard identities \cite[Section 2.3.4]{Edwards1}, the derivatives $h'$ and $h''$ are
\begin{equation}\label{h'}
    h'(\theta)=\sum_{k=1}^\infty k(-a_k\sin k\theta+b_k\cos k\theta),
\end{equation}
\begin{equation}\label{h''}
h''(\theta)=\sum_{k=1}^\infty k^2(-a_k\cos k\theta-b_k\sin k\theta).
\end{equation}

\begin{explanation}
These identities hold, for instance, in $L^2(0,2\pi)$, which is sufficient for the Parseval computations below. 
\end{explanation}

Our goal is to see how the constant width condition and the group symmetry restrict the above coefficients. 

\begin{proposition} \label{prop:coeff}
If $h\in C^2$ and $h(\theta) + h(\theta+\pi)=w$, then:
\begin{enumerate}
\item $a_0=\frac{w}{2}$,
\item $a_k=b_k=0$ for all even $k\ge 2$.
\end{enumerate}
\end{proposition}

\begin{proof}
For $k \in \mathbb{N}$, we use identities
\(
\cos(k(\theta+\pi))=
(-1)^k\cos(k\theta),
\sin(k(\theta+\pi))=
(-1)^k\sin(k\theta).
\)
Thus, we have
\begin{align*}
h(\theta)+h(\theta+\pi)&=a_0+\sum_{k=1}^\infty(a_k\cos k\theta+b_k\sin k\theta)+a_0+\sum_{k=1}^\infty(a_k\cos k(\theta+\pi)+b_k\sin k(\theta+\pi))\\
&=2a_0+\sum_{k=1}^\infty(a_k\cos k\theta+b_k\sin k\theta)+\sum_{k=1}^\infty((-1)^ka_k\cos k\theta+(-1)^kb_k\sin k\theta)\\
&=2a_0+\sum_{k=1}^\infty(1+(-1)^k)(a_k\cos k\theta+b_k\sin k\theta).
\end{align*}

For this to equal the constant $w$ for all $\theta$, we need:
\begin{itemize}
\item The constant term: $2a_0=w$, so $a_0=\frac{w}{2}$.
\item For each $k\ge 1$, the coefficient of $\cos k\theta$ and $\sin k\theta$ must be zero.
\end{itemize}
If $k$ is odd, then $1+(-1)^k=0$. The condition is automatically satisfied for any $a_k,b_k$.
If $k$ is even, then $1+(-1)^k=2$. We need $2a_k=0$ and $2b_k=0$, hence $a_k=b_k=0$.
\end{proof}

\begin{corollary}
For a constant-width body, the support function has the form
\[
h(\theta) = \frac{w}{2}
+ \sum_{\substack{k\ge1\\ k\text{ odd}}}
(a_k\cos k\theta + b_k\sin k\theta).
\]
\end{corollary}

If $\mathcal R_{2\pi/n}K=K$, we get that $h$ is $2\pi/n$-periodic,
\begin{explanation}
\[
h(\theta)=h\Bigl(\theta-\frac{2\pi}{n}\Bigr)
\qquad\text{for all }\theta,
\]
and therefore, after replacing $\theta$ by $\theta+\frac{2\pi}{n}$,
\end{explanation}
\begin{equation}\label{period}
  h\Bigl(\theta+\frac{2\pi}{n}\Bigr)=h(\theta)
\qquad\text{for all }\theta.  
\end{equation}
This periodicity condition imposes strong restrictions on the Fourier coefficients of $h$.

\begin{proposition} \label{prop:divisible}
If a $2\pi$-periodic function $h$ satisfies \eqref{period},
then in its Fourier expansion
\[
h(\theta)=a_0+\sum_{k=1}^{\infty}(a_k\cos k\theta+b_k\sin k\theta),
\]
one has $a_k=b_k=0$ whenever $n\nmid k$. 
\end{proposition}

\begin{proof}
Fix $k\ge1$. Since $h$ is $2\pi/n$-periodic, the $k$-th cosine coefficient satisfies
\begin{align*}
a_k
&=\frac{1}{\pi}\int_0^{2\pi} h(\theta)\cos(k\theta)\,\dd\theta=\frac{1}{\pi}\int_0^{2\pi} h\Bigl(\theta+\frac{2\pi}{n}\Bigr)\cos(k\theta)\,\dd\theta.
\end{align*}
With the change of variables $\phi=\theta+\frac{2\pi}{n}$ and using the periodicity of $h$, this becomes
\begin{align*}
a_k
&=\frac{1}{\pi}\int_0^{2\pi} h(\phi)\cos\Bigl(k\phi-\frac{2\pi k}{n}\Bigr)\,\dd\phi
=\cos\Bigl(\frac{2\pi k}{n}\Bigr)a_k+\sin\Bigl(\frac{2\pi k}{n}\Bigr)b_k.
\end{align*}
Similarly,
\[
b_k=-\sin\Bigl(\frac{2\pi k}{n}\Bigr)a_k+\cos\Bigl(\frac{2\pi k}{n}\Bigr)b_k.
\]
Hence the vector $(a_k,b_k)$ is fixed by the planar rotation through angle $\frac{2\pi k}{n}$. If $n\nmid k$, then this angle is not a multiple of $2\pi$, and the only fixed vector of a nontrivial planar rotation is $(0,0)$. Therefore $a_k=b_k=0$.
If $n\mid k$, no such vanishing is forced. Thus only the modes $k=mn$, $m \in \mathbb{Z}$, can occur.
\end{proof}

By Propositions \ref{prop:coeff} and \ref{prop:divisible}, if $K$ is a constant-width body invariant under $C_n$ or $D_n$ for some even $n$, then $K$ must be a disk as mentioned in the introduction, and if $n$ is odd, then it can only have odd modes that are multiples of $n$ in their Fourier expansion.

\begin{proposition}
    \label{prop:even-odd-modes}
If $K$ is a constant-width body invariant under $C_n$ or $D_n$ then:
\begin{itemize}
    \item If $n$ is \textbf{even}, then every admissible frequency $k=mn$ is even. But constant width forces every even Fourier mode to vanish. Hence all nonconstant modes disappear, so necessarily $h(\theta)\equiv w/2$. As a consequence, $K$ is a disk of radius $w/2$.
    \item If $n$ is \textbf{odd}, the only frequencies compatible with both conditions are the odd multiples of $n$, namely $k=n,3n,5n,\dots$.
\end{itemize}
\end{proposition}

Having disposed of the exceptional case $D_1$ in Proposition \ref{prop:D1} and of the even case in Proposition \ref{prop:even-odd-modes}, we henceforth assume that $n\ge3$ is odd and that $K$ is $C_n$-invariant. Accordingly,
\[
h(\theta)=\frac{w}{2}+\sum_{\substack{m\ge1\\m\text{ odd}}}\bigl(a_{mn}\cos(mn\theta)+b_{mn}\sin(mn\theta)\bigr).
\]
\begin{explanation}
    so every nonconstant Fourier mode is written as an odd multiple of $n$.
\end{explanation}

\section{Proof of the odd case}
We start by expressing the area of $K$ in terms of the Fourier coefficients of $h$.

\begin{theorem}[Area in Fourier Variables]
\label{thm:F-area}
    Fix an odd integer $n\ge3$, and assume that the support function $h$ of $K$ belongs to $C^2$. For a $C_n$-invariant constant-width body with Fourier coefficients $a_{mn}, b_{mn}$, the area is given by
\[
\operatorname{Area}(K)
=
\frac{\pi w^2}{4}
-
\frac{\pi}{2}
\sum_{\substack{m\ge1\\ m\text{ odd}}}
\bigl((mn)^2-1\bigr)(a_{mn}^2+b_{mn}^2).
\]
\end{theorem}

\begin{proof}
  To apply Theorem \ref{thm:area-formula}, let us start by computing $\int_0^{2\pi}h^2\dd\theta$.
Using Parseval's identity,
\[
\int_0^{2\pi}h^2\dd\theta=2\pi a_0^2+\pi\sum_{k=1}^\infty(a_k^2+b_k^2).
\]

By Proposition \ref{prop:even-odd-modes}, for an odd $C_n$-invariant constant-width body, only the coefficients with indices $mn$ ($m$ odd) may be nonzero,
\[
\int_0^{2\pi}h^2\dd\theta
=2\pi\left(\frac{w}{2}\right)^2+\pi\sum_{\substack{m\ge1\\m\text{ odd}}}(a_{mn}^2+b_{mn}^2)
=\frac{\pi w^2}{2}+\pi\sum_{\substack{m\ge1\\m\text{ odd}}}(a_{mn}^2+b_{mn}^2).
\]

Now we compute $\int_0^{2\pi}(h')^2\dd\theta$. 
Using \eqref{h'} and Parseval's identity again,
\[
\int_0^{2\pi}(h')^2\dd\theta=\pi\sum_{\substack{m\ge1\\m\text{ odd}}}(mn)^2(a_{mn}^2+b_{mn}^2).
\]

Thus, by Theorem \ref{thm:area-formula}, the area is
\begin{align*}
\operatorname{Area}(K)&=\frac{1}{2}\int_0^{2\pi}(h^2-(h')^2)\dd\theta\\
&=\frac{1}{2}\left[\frac{\pi w^2}{2}+\pi\sum_{\substack{m\ge1\\m\text{ odd}}}(a_{mn}^2+b_{mn}^2)-\pi\sum_{\substack{m\ge1\\m\text{ odd}}}(mn)^2(a_{mn}^2+b_{mn}^2)\right]\\
&=\frac{\pi w^2}{4}-\frac{\pi}{2}\sum_{\substack{m\ge1\\m\text{ odd}}}\bigl((mn)^2-1\bigr)(a_{mn}^2+b_{mn}^2).
\end{align*}

\end{proof}

Hence, to minimize the $\operatorname{Area}(K)$ among all constant-width bodies of width $w$, we need to maximize the deficit term,
\[
\frac{\pi}{2} \sum_{\substack{m \ge 1\\m\text{ odd}}}((mn)^2-1)(a_{mn}^2+b_{mn}^2)
\]
subject to
\begin{enumerate}
\item $h$ has the Fourier form $h(\theta)=\frac{w}{2}+\sum_{m\text{ odd}, m\ge 1} (a_{mn}\cos (mn\theta)+b_{mn}\sin (mn\theta))$,
\item $h(\theta)+h''(\theta)\ge 0$ for all $\theta$.
\end{enumerate}
Working directly with $h$ is difficult because the coefficients $a_{mn}, b_{mn}$ are coupled by the pointwise inequality constraint involving the second derivative; therefore, it is much more effective to work with the radius of curvature $\rho$.

\begin{lemma}
    Fix an odd integer $n\ge3$. For a $C_n$-invariant constant-width body with smooth support function $h$ and curvature radius $\rho$, the Fourier coefficients of $\rho$, denoted $\alpha_{mn}, \beta_{mn}$, are related to those of $h$, denoted by $a_{mn}, b_{mn}$, by 
\[
\alpha_{mn} = (1-(mn)^2)a_{mn}, \quad \beta_{mn} = (1-(mn)^2)b_{mn}.
\]
The constant term is $\alpha_0 = \frac{w}{2}$.
\end{lemma}

\begin{proof}
By \eqref{rho} and \eqref{h''},
\[
\rho(\theta) = \frac{w}{2} + \sum (1-(mn)^2)(a_{mn}\cos (mn\theta) + b_{mn} \sin (mn\theta)).
\]
Identifying coefficients yields the conclusion.
\end{proof}

Substituting $a_{mn} = \frac{\alpha_{mn}}{1-(mn)^2}$ and  $b_{mn} = \frac{\beta_{mn}}{1-(mn)^2}$ into the deficit formula our problem becomes to maximize
\begin{equation} \label{J}
    J(\rho) := \frac{\pi}{2} \sum_{\substack{m \ge 1\\m\text{ odd}}} \frac{\alpha_{mn}^2+\beta_{mn}^2}{(mn)^2-1},
\end{equation}
subject to the constraints on $\rho$:
\begin{enumerate}
\item convexity: $\rho(\theta) \ge 0$ for all $\theta$,
\item constant width: $\rho(\theta) + \rho(\theta+\pi) = w$ for all $\theta$,
\item $C_n$-invariance: $\rho(\theta + 2\pi/n) = \rho(\theta)$ for all $\theta$.
\end{enumerate}


Up to this point, $\rho$ has been interpreted under the temporary assumption $h\in C^2$. The next proposition shows that for constant-width bodies with nontrivial rotational symmetry the same object is available without any smoothness hypothesis. 


\begin{proposition}[Curvature density for constant-width bodies]
\label{prop:ac-curvature}

Let $K\subset\mathbb{R}^{2}$ be a planar convex body of constant
width $w>0$, and let $h$ be its support function. Define the periodic
distribution
\[
    \mu:=h+h''
\]
on $\mathbb{S}^{1}$ by
\[
    \langle \mu,\varphi\rangle
    :=
    \langle h+h'',\varphi\rangle
    =
    \int_{0}^{2\pi}
        h(\theta)
        \bigl(\varphi(\theta)+\varphi''(\theta)\bigr)
    \,d\theta,
    \qquad
    \varphi\in C^{\infty}(\mathbb{S}^{1}).
\]
Then $\mu$ is a nonnegative finite Borel measure on $\mathbb{S}^{1}$.
Moreover, $\mu$ is absolutely continuous with respect to the
Lebesgue measure $d\theta$. More precisely, there exists
$\rho\in L^{\infty}(\mathbb{S}^{1})$ such that
\[
    \mu=\rho(\theta)\,d\theta,
\]
and
\[
    0\leq \rho(\theta)\leq w,
    \qquad
    \rho(\theta)+\rho(\theta+\pi)=w
\]
for almost every $\theta\in\mathbb{S}^{1}$.

If, in addition, $K$ is invariant under the cyclic group $C_n$,
then
\[
    \rho\left(\theta+\frac{2\pi}{n}\right)
    =
    \rho(\theta)
\]
for almost every $\theta\in\mathbb{S}^{1}$.
\end{proposition}

\begin{proof}
First, we identify \(\mathbb S^1\) with \(\mathbb{R}/2\pi\mathbb{Z}\), and let
\(d\theta\) denote the standard Lebesgue measure on \(\mathbb S^1\). For an
arbitrary planar convex body, the periodic distribution
$
\mu
$
is the first-order area measure of \(K\); in particular, it is a
nonnegative finite Borel measure on \(\mathbb S^1\) (see
\cite[Theorem~4.2.1 
]{Schneider2014}).
For \(a\in\mathbb{R}\), let \(\tau_a\) denote translation by \(a\), so
that
\[
(\tau_a f)(\theta):=f(\theta+a)
\]
for periodic functions, with the corresponding extension to
distributions and measures. Since \(K\) has constant width \(w\),
\[
h+\tau_\pi h=w.
\]
Applying the operator \(1+\partial_\theta^2\) in the sense of
distributions, and using the fact that translations commute with
distributional differentiation, gives
\[
\mu+\tau_\pi\mu=w\,d\theta.
\]
Let
\[
\mu=\rho\,d\theta+\mu_s
\]
be the Lebesgue decomposition of \(\mu\) with respect to \(d\theta\),
where \(\rho\in L^1(\mathbb S^1)\) and \(\mu_s\) is singular with respect to $d\theta$. Translation
preserves both absolute continuity and singularity with respect to
\(d\theta\), and hence
\[
\tau_\pi\mu
=
\rho(\theta+\pi)\,d\theta+\tau_\pi\mu_s.
\]
By uniqueness of the Lebesgue decomposition, the singular part of the
identity
\[
\mu+\tau_\pi\mu=w\,d\theta
\]
satisfies
\(
\mu_s+\tau_\pi\mu_s=0.
\)
Since both \(\mu_s\) and \(\tau_\pi\mu_s\) are nonnegative measures, it
follows that
$
\mu_s=0.
$
Therefore,
$
\mu=\rho\,d\theta
$
for some nonnegative \(\rho\in L^1(\mathbb S^1)\). The absolutely continuous
part of the identity above is then
\[
\bigl(\rho(\theta)+\rho(\theta+\pi)\bigr)\,d\theta
=
w\,d\theta.
\]
By uniqueness of Radon--Nikodym derivatives,
\(
\rho(\theta)+\rho(\theta+\pi)=w
\)
for almost every $\theta$.
Since both terms on the left-hand side are nonnegative, we obtain
\[
0\leq \rho(\theta)\leq w
\qquad\text{for almost every }\theta.
\]
Consequently,
$
\rho\in L^\infty(\mathbb S^1).
$
Finally, the \(C_n\)-invariance of \(K\) implies
$
\tau_{2\pi/n}h=h.
$
Applying \(1+\partial_\theta^2\) in the sense of distributions yields
\[
\tau_{2\pi/n}\mu=\mu.
\]
Since \(\mu=\rho\,d\theta\), uniqueness of the Radon--Nikodym derivative
gives
\[
\rho\left(\theta+\frac{2\pi}{n}\right)=\rho(\theta)
\qquad\text{for almost every }\theta.
\]
\end{proof}

\begin{explanation}
By Proposition \ref{prop:ac-curvature}, $\rho$ the average value of the admissible profiles $\rho$ is exactly $w/2$, and
\[
0 \le \rho(\theta) \le w \quad \text{for almost every } \theta.
\]
\end{explanation}
Define an admissible class of profiles $\rho$,
\[
\mathcal A_w^{C_n}\coloneqq
\Bigl\{\rho\in L^\infty(0,2\pi):
0\le \rho\le w,\;
\rho(\theta)+\rho(\theta+\pi)=w\text{ a.e.},\;
\rho\Bigl(\theta+\frac{2\pi}{n}\Bigr)=\rho(\theta)\text{ a.e.}
\Bigr\}.
\]
By Proposition \ref{prop:ac-curvature}, every $C_n$-invariant planar convex body of constant width $w$ determines an element of $\mathcal A_w^{C_n}$. The next lemma shows the converse and provides the corresponding area identity in terms of $\eqref{J}$. 
\begin{explanation}
    It begins the reconstruction step, and the subsequent corollary records the resulting equivalence with the original geometric problem.
\end{explanation}

\begin{lemma}[Reconstruction from $\rho$]
\label{lem:reconstruct-rho}
Let $\alpha_{mn}$ and $\beta_{mn}$ be the Fourier coefficients of $\rho\in\mathcal A_w^{C_n}$. Then the Fourier series of $\rho$, interpreted distributionally, is
\[
\rho(\theta)=\frac w2+\sum_{\substack{m\ge1\\m\text{ odd}}}
\bigl(\alpha_{mn}\cos(mn\theta)+\beta_{mn}\sin(mn\theta)\bigr).
\]
Let
\[
h(\theta)=\frac w2+\sum_{\substack{m\ge1\\m\text{ odd}}}
\frac{\alpha_{mn}\cos(mn\theta)+\beta_{mn}\sin(mn\theta)}{1-(mn)^2}.
\]
Then the series for $h$ converges absolutely and uniformly, one has
\[
h\Bigl(\theta+\frac{2\pi}{n}\Bigr)=h(\theta),
\qquad
h(\theta+\pi)+h(\theta)=w,
\]
and $h+h''=\rho$ in the sense of distributions. Moreover, $h$ is the support function of a $C_n$-invariant planar convex body $K$ of constant width $w$, and
\begin{equation} \label{Area}
    \Area(K)=\frac{\pi w^2}{4}-J(\rho).
\end{equation}

\end{lemma}

\begin{proof}
For every odd $m\ge1$, the corresponding Fourier coefficients can be written as
\[
\alpha_{mn}=\frac{1}{\pi}\int_0^{2\pi}\Bigl(\rho(\theta)-\frac w2\Bigr)\cos(mn\theta)\,\dd\theta,
\qquad
\beta_{mn}=\frac{1}{\pi}\int_0^{2\pi}\Bigl(\rho(\theta)-\frac w2\Bigr)\sin(mn\theta)\,\dd\theta,
\]
because the nonconstant trigonometric modes have zero mean. Since $0\le \rho\le w$, one has $\bigl|\rho-\frac w2\bigr|\le \frac w2$, and therefore
$
|\alpha_{mn}|,|\beta_{mn}|\le w.
$
Hence,
\[
\sum_{\substack{m\ge1\\m\text{ odd}}}
\frac{|\alpha_{mn}|+|\beta_{mn}|}{(mn)^2-1}
\le
2w\sum_{\substack{m\ge1\\m\text{ odd}}}
\frac{1}{(mn)^2-1}
<\infty,
\]
so the series defining $h$ converges absolutely and uniformly. The frequency restrictions immediately imply the displayed $C_n$-symmetry and constant-width identities, and termwise differentiation in the sense of distributions gives $h+h''=\rho$.

Let $\rho_N(\theta) = (F_N*\rho)(\theta) = \int_0^{2\pi}F_N(\theta-t)\rho(t)\,\dd t$ be the $N$th Fej\'er mean for Fej\'er kernel \cite[Section 5.1]{Edwards1},
$$
F_N(x) = \frac{1}{2 \pi} \frac{D_0(x)+\ldots+D_N(x)}{N+1}, \quad \text{where} \quad D_N(x) = \sum_{|k| \leq N} e^{ikx}.
$$
Because $F_N$ is nonnegative, has total mass $1$, and convolution with $F_N$ commutes with angular translations, one has for every shift $\alpha$,
\[
\rho_N(\theta+\alpha)=F_N*(\rho(\cdot+\alpha))(\theta).
\]
Therefore the identities satisfied by $\rho$ are inherited by $\rho_N$. In particular, each $\rho_N$ still satisfies
\[
0\le \rho_N\le w,
\qquad
\rho_N(\theta+\pi)+\rho_N(\theta)=w,
\qquad
\rho_N\Bigl(\theta+\frac{2\pi}{n}\Bigr)=\rho_N(\theta).
\]
Define $h_N$ from the Fourier coefficients of $\rho_N$ by the same formula as above. Then $h_N$ is $C^2$-smooth, satisfies $h_N+h_N''=\rho_N\ge0$. By the characterization proved earlier in Lemma \ref{lemmaConvexity}, $h_N$ is the support function of a convex body $K_N$. Similar computations to the previous ones show that $h_N(\theta)+h_N(\theta + \pi)=w$ and $h_N(\theta + \frac{2\pi}{n})=h_N(\theta)$, hence, $K_N$ has constant width $w$ and is $C_n$-invariant.

For each frequency $k$, the $k$th Fourier coefficient of $\rho_N$ is $\lambda_{N,k}$ times the corresponding coefficient of $\rho$, where $0\le\lambda_{N,k}\le1$ and $\lambda_{N,k}\to1$. Since the series defining $h$ is absolutely convergent, dominated convergence shows that $h_N\to h$ uniformly. A uniform limit of support functions is again a support function, so $h$ is the support function of a convex body $K$. The width and symmetry relations pass to the limit, and uniform convergence of support functions is equivalent to Hausdorff convergence of the associated convex bodies; see, for instance, \cite[Chapter 1]{Schneider2014}. Thus $K_N\to K$ in the Hausdorff metric.
For each $N$, we apply Theorem \ref{thm:F-area} and rewrite the area in terms of the Fourier coefficients of $\rho_N$ to obtain
\[
\Area(K_N)=\frac{\pi w^2}{4}-J(\rho_N).
\]
Because $|\alpha_{mn}(\rho_N)|,|\beta_{mn}(\rho_N)|\le w$ for all $m,N$ and the Fourier coefficients of $\rho_N$ converge to those of $\rho$, dominated convergence in the defining series of $J$ shows that
\[
J(\rho_N)\to J(\rho).
\]
Since area is continuous with respect to Hausdorff convergence on planar convex bodies; see \cite[Chapter 1]{Schneider2014},
\[
\Area(K_N)\to \Area(K).
\]
Passing to the limit in the previous identity gives \eqref{Area}.
\end{proof}

\begin{corollary}
\label{cor:no-relaxation}
For odd $n\ge3$, the map sending a body $K$ to the density $\rho$ determined by
\[
h+h''=\rho(\theta)\,\dd\theta
\]
gives a bijection between the $C_n$-invariant planar convex bodies of constant width $w$ and the admissible class $\mathcal A_w^{C_n}$. In particular, minimizing area over such bodies is equivalent to maximizing $J$ over $\mathcal A_w^{C_n}$.
\end{corollary}

\begin{proof}
Existence of $\rho$ for a given body follows from Proposition \ref{prop:ac-curvature}, and Lemma \ref{lem:reconstruct-rho} reconstructs a $C_n$-invariant planar convex body of constant width $w$ from any profile in $\mathcal A_w^{C_n}$.

To prove uniqueness, let $h_1$ and $h_2$ be two $C_n$-invariant support functions of constant width $w$ with the same profile $\rho$. Then $q\coloneqq h_1-h_2$ satisfies $q+q''=0$ in the sense of distributions, so
$
q(\theta)=a\cos\theta+b\sin\theta
$
for some $a,b\in\R$. Since $q\bigl(\theta+\frac{2\pi}{n}\bigr)=q(\theta)$ and $n\ge3$, the vector $(a,b)$ is fixed by the nontrivial rotation through angle $2\pi/n$, hence $a=b=0$, and $h_1\equiv h_2$.
Therefore each admissible body is uniquely determined by its profile, and Lemma \ref{lem:reconstruct-rho} gives \eqref{Area}.
\end{proof}

Since the mean value of $\rho$ is $w/2$, it is more convenient to work with the zero-mean fluctuation
\begin{equation*} 
    f(\theta)\coloneqq \rho(\theta)-\frac w2.
\end{equation*}
Then
\[
f(\theta+\pi)=-f(\theta),\qquad |f(\theta)|\le \frac w2,
\qquad f\Bigl(\theta+\frac{2\pi}{n}\Bigr)=f(\theta).
\]
For $f$, let us define
\begin{equation}\label{I}
\mathcal J(f) := J(\rho)=\frac{\pi}{2}\sum_{\substack{m\ge1\\m\text{ odd}}}
\frac{\alpha_{mn}^2+\beta_{mn}^2}{(mn)^2-1}, 
\end{equation}
and the admissible class of zero-mean fluctuations corresponding to $\mathcal A_w^{C_n}$ by
\[
\mathcal F_w^{C_n}
\coloneqq
\Bigl\{f\in L^\infty(0,2\pi):
|f|\le \tfrac w2,\;
f(\theta+\pi)=-f(\theta)\text{ a.e.},\;
f\Bigl(\theta+\tfrac{2\pi}{n}\Bigr)=f(\theta)\text{ a.e.}
\Bigr\}.
\]
Recall that our goal is to maximize the functional $\mathcal J(f)$
over $\mathcal F_w^{C_n}$. To analyze this maximization problem we rewrite the objective functional $\mathcal J(f)$ as an integral operator.

\begin{explanation}

\begin{theorem}[Convolution Theorem for Fourier Series]
Let $f$ and $g$ be $2\pi$-periodic functions with Fourier series
\[
f(\theta) = a_0 + \sum_{k=1}^\infty (a_k\cos k\theta + b_k\sin k\theta), \quad g(\theta) = \alpha_0 + \sum_{k=1}^\infty (\alpha_k\cos k\theta + \beta_k\sin k\theta).
\]
Their convolution is defined as $(f*g)(\theta) = \int_0^{2\pi} f(\phi)g(\theta-\phi)\,\dd\phi$. The Fourier series of the convolution is given by:
\[
(f*g)(\theta) = 2\pi a_0\alpha_0 + \pi \sum_{k=1}^\infty \bigl[ (a_k\alpha_k - b_k\beta_k)\cos k\theta + (a_k\beta_k + b_k\alpha_k)\sin k\theta \bigr].
\]
\end{theorem}

\begin{theorem}[Parseval's Identity for Inner Products]
For $2\pi$-periodic functions $f$ and $g$ as defined above, their $L^2$ inner product can be computed from their Fourier coefficients:
\[
\int_0^{2\pi} f(\theta)g(\theta)\,\dd\theta = 2\pi a_0\alpha_0 + \pi \sum_{k=1}^\infty (a_k\alpha_k + b_k\beta_k).
\]
\end{theorem}
\end{explanation}


\begin{proposition}[Integral representation of the objective $\mathcal J$]
\label{prop:objective-kernel}
Define the even convolution kernel
\begin{equation} \label{K}
M(\theta) = \sum_{\substack{m\ge1\\m\text{ odd}}}\frac{\cos(mn\theta)}{(mn)^2-1}.  
\end{equation}

For the Fourier expansion of $f\in \mathcal F_w^{C_n}$,
$
f(\theta)=\sum_{\substack{m\ge1\\ m\text{ odd}}}\bigl(\alpha_{mn}\cos(mn\theta)+\beta_{mn}\sin(mn\theta)\bigr),
$
 let $\Psi_f\coloneqq M*f$. Then
\[
\Psi_f(\theta)=\pi \sum_{\substack{m\ge1\\m\text{ odd}}} \frac{\alpha_{mn}\cos (mn\theta) + \beta_{mn}\sin (mn\theta)}{(mn)^2-1},
\]
and
\begin{equation}\label{II}
    \mathcal J(f) = \frac{1}{2\pi}\int_0^{2\pi} f(\theta)\,\Psi_f(\theta)\,\dd\theta.
\end{equation}

\end{proposition}

\begin{proof}
The only nonzero Fourier coefficients of $M$ are
\(
c_{mn}=\frac{1}{(mn)^2-1}.
\)
Therefore, 
\begin{explanation}
by the Convolution Theorem for Fourier series,  
\end{explanation}
\[
(M*f)(\theta) = \int_0^{2\pi} M(\theta-\phi)f(\phi)\,\dd\phi
= \pi \sum_{\substack{m\ge1\\m\text{ odd}}} \frac{\alpha_{mn}\cos (mn\theta) + \beta_{mn}\sin (mn\theta)}{(mn)^2-1}.
\]
Applying Parseval's Identity to the pair $f$ and $M*f$ gives \eqref{I},
\[
\frac{1}{2\pi}\int_0^{2\pi} f(\theta)\,(M*f)(\theta)\,\dd\theta
=
\frac{\pi}{2}\sum_{\substack{m\ge1\\m\text{ odd}}}
\frac{\alpha_{mn}^2+\beta_{mn}^2}{(mn)^2-1}
=\mathcal J(f).
\]
\end{proof}

The next result establishes the optimal profiles for $\mathcal J$ over $\mathcal F_w^{C_n}$.
\begin{proposition}[Maximizers are bang-bang]
\label{prop:bang-bang-maximizer}
The functional $\mathcal J$ attains its maximum on $\mathcal F_w^{C_n}$. Moreover, every maximizer $f_*$ satisfies
\[
|f_*(\theta)|=\frac w2 \quad\text{for almost every }\theta.
\]
\end{proposition}

\begin{proof}
The admissible set $\mathcal F_w^{C_n}$ is convex and weak-* compact in $L^\infty(0,2\pi)$, because it is a weak-* closed subset of the closed ball $\{f\in L^\infty:\|f\|_\infty\le w/2\}$. If $f_j\stackrel{*}{\rightharpoonup}f$ in $L^\infty$, then every Fourier coefficient of $f_j$ converges to the corresponding Fourier coefficient of $f$. Since, $\sum_{\substack{m\ge1\\m\text{ odd}}}\frac{1}{(mn)^2-1}<\infty$, and
\[
|\alpha_{mn}(f_j)|,|\beta_{mn}(f_j)|\le w
\qquad\text{for all }j,m \in \mathbb{N},
\]
the dominated convergence in the defining series shows that $\mathcal J(f_j)\to \mathcal J(f)$. Hence $\mathcal J$ is weak-* continuous, and therefore it attains its maximum on the weak-* compact set $\mathcal F_w^{C_n}$.

We next show that every maximizer is an extreme point of $\mathcal F_w^{C_n}$. The coefficients in the defining series \eqref{I} are strictly positive, so $\mathcal J$ is a strictly convex quadratic form on the symmetry subspace underlying $\mathcal F_w^{C_n}$. Concretely, if $f_1\neq f_2$ belong to $\mathcal F_w^{C_n}$, then their difference has at least one nonzero admissible Fourier coefficient, and therefore
\[
\mathcal J\Bigl(\frac{f_1+f_2}{2}\Bigr)
=
\frac{\mathcal J(f_1)+\mathcal J(f_2)}{2}
-
\mathcal J\Bigl(\frac{f_1-f_2}{2}\Bigr)
<
\frac{\mathcal J(f_1)+\mathcal J(f_2)}{2}.
\]
If a maximizer $f$ were not extreme, say $f=(f_1+f_2)/2$ with distinct $f_1,f_2\in\mathcal F_w^{C_n}$, then the previous strict inequality would contradict maximality. 
Thus, every maximizer is extreme.

It remains to identify the extreme points of $\mathcal F_w^{C_n}$. We will show that the extreme points are exactly the bang-bang functions, i.e., those satisfying $|f(\theta)|=w/2$ almost everywhere.
First, if $|f(\theta)|=w/2$ almost everywhere and
\[
f=\frac{f_1+f_2}{2}
\qquad\text{with }f_1,f_2\in\mathcal F_w^{C_n},
\]
then on the set where $f=w/2$ we must have $f_1=f_2=w/2$ almost everywhere, and similarly on the set where $f=-w/2$. Thus $f_1=f_2=f$, so $f$ is extreme.

Conversely, assume that $f\in\mathcal F_w^{C_n}$ is not bang-bang. Then there exists $\delta>0$ such that the set
\[
E_\delta\coloneqq\Bigl\{\theta:\ |f(\theta)|\le \frac w2-\delta\Bigr\}
\]
has positive measure. Because $f\bigl(\theta+\frac{2\pi}{n}\bigr)=f(\theta)$ and $f(\theta+\pi)=-f(\theta)$ hold almost everywhere, the set $E_\delta$ is invariant under the shifts $\theta\mapsto\theta+\frac{2\pi}{n}$ and $\theta\mapsto\theta+\pi$ modulo null sets. Let
\[
\Sigma\coloneqq\Bigl\{\frac{2k\pi}{n},\ \pi+\frac{2k\pi}{n}:0\le k\le n-1\Bigr\}.
\]
 Because $n$ is odd, the $2n$ points of $\Sigma$ are distinct modulo $2\pi$. Choose an interval $I$ so short that the translates of $I$ by these $2n$ shifts are pairwise disjoint modulo $2\pi$, and choose a measurable set $A\subset  E_\delta\cap I$ of positive measure. Define
\[
\eta(\theta)
=
\sum_{k=0}^{n-1}\chi_{A+2k\pi/n}(\theta)
-
\sum_{k=0}^{n-1}\chi_{A+\pi+2k\pi/n}(\theta).
\]
Then $\eta\not\equiv0$, $|\eta|\le1$, $\operatorname{supp}(\eta)\subset E_\delta$, and
\[
\eta\Bigl(\theta+\frac{2\pi}{n}\Bigr)=\eta(\theta),
\qquad
\eta(\theta+\pi)=-\eta(\theta).
\]
Therefore, for every $0<\varepsilon<\delta$ one has $f\pm\varepsilon\eta\in\mathcal F_w^{C_n}$, and these two functions are distinct because $\eta\not\equiv0$. Since
\[
f=\frac{(f+\varepsilon\eta)+(f-\varepsilon\eta)}{2},
\]
the function $f$ is not an extreme point. We conclude that the extreme points of $\mathcal F_w^{C_n}$ are exactly the bang-bang functions. Since every maximizer is extreme, every maximizer $f_*$ satisfies $|f_*|=w/2$ almost everywhere.
\end{proof}

To pass from the functional $\mathcal J(f)$ to a set-rearrangement problem on the standard circle, we rescale the angular variable by $x=n\theta$. Under this change of variables, the kernel from  \eqref{K}
becomes the $2\pi$-periodic kernel
\begin{equation}
\label{HnKernel}
H_n(t)\coloneqq \sum_{\substack{m\ge1\\m\text{ odd}}}\frac{\cos(mt)}{n^2m^2-1},
\qquad t\in\mathbb R,
\end{equation}
so that $M(\theta)=H_n(n\theta)$. The point of introducing $H_n$ is that the following rearrangement argument is carried out in the $x$-variable on the standard circle, and for that we need to know that the relevant kernel is explicit, even, and strictly decreasing on $[0,\pi]$. The next lemma proves exactly this.

\begin{lemma}[Explicit form of the rearrangement kernel]\label{lem:Hn-kernel}
For \eqref{HnKernel} and $0\le t\le \pi$, one has
\[
H_n(t)=\frac{\pi}{4n\cos\!\left(\frac{\pi}{2n}\right)}\sin\!\left(\frac{\pi-2t}{2n}\right).
\]
In particular, $H_n$ is even and strictly decreasing on $(0,\pi)$.
\end{lemma}

\begin{proof}
Since
$
\sum_{\substack{m\ge1\\m\text{ odd}}}\frac{1}{n^2m^2-1}<\infty,
$
the defining Fourier series of $H_n$ converges absolutely and uniformly, so $H_n$ is continuous and even. Because only odd frequencies occur, it also satisfies
\[
H_n(t+\pi)=-H_n(t)
\qquad\text{for all }t.
\]
Note that in the sense of distributions on the circle,
$
\bigl(1+n^2\partial_t^2\bigr)H_n
=-\sum_{\substack{m\ge1\\m\text{ odd}}}\cos(mt).
$
Using the Fourier expansion of the periodic Dirac masses at $0$ and $\pi$,
$
\delta_0-\delta_\pi=\frac{2}{\pi}\sum_{\substack{m\ge1\\m\text{ odd}}}\cos(mt),
$
we obtain
\[
\bigl(1+n^2\partial_t^2\bigr)H_n=-\frac{\pi}{2}(\delta_0-\delta_\pi).
\]
Therefore, $H_n$ solves the homogeneous ODE
\[
H_n''+\frac{1}{n^2}H_n=0
\]
on the open interval $(0,\pi)$. Hence, there are constants $A,B$ such that
$
H_n(t)=A\cos\frac{t}{n}+B\sin\frac{t}{n}
$
for $0<t<\pi$.
Since $H_n$ is even and $\pi$-antiperiodic, we have
$
H_n(\pi-t)=H_n(t-\pi)=-H_n(t).
$
In particular $H_n(\pi/2)=0$. Thus, 
\(
H_n(t)=C\sin\!\left(\frac{\pi-2t}{2n}\right)
\)
for $0<t<\pi$ and some constant $C$.
To determine $C$, integrate the distributional identity across $t=0$,
\[
n^2\bigl(H_n'(0^+)-H_n'(0^-)\bigr)=-\frac{\pi}{2}.
\]
Since $H_n$ is even, $H_n'$ is odd, so $H_n'(0^-)=-H_n'(0^+)$. Hence
\(
H_n'(0^+)=-\frac{\pi}{4n^2}.
\)
On the other hand, differentiating the explicit form above gives
\(
H_n'(0^+)=-\frac{C}{n}\cos\!\left(\frac{\pi}{2n}\right).
\)
Comparing the two expressions yields
\(
C=\frac{\pi}{4n\cos\!\left(\frac{\pi}{2n}\right)}.
\)
This proves the claimed formula for $(0,\pi)$, and continuity extends it to $[0,\pi]$. Finally, for $0<t<\pi$,
\[
H_n'(t)=-\frac{\pi}{4n^2\cos\!\left(\frac{\pi}{2n}\right)}\cos\!\left(\frac{\pi-2t}{2n}\right)<0,
\]
because $\frac{\pi-2t}{2n}\in\bigl(-\frac{\pi}{2n},\frac{\pi}{2n}\bigr)$ and $n\ge3$. Thus, $H_n$ is strictly decreasing on $(0,\pi)$.
\end{proof}

We proceed to the rearrangement argument on the standard circle,
$
\mathbb T:= \mathbb R/2\pi\mathbb Z,
$
that we identify with the unit circle $\mathbb{S}^1$ via $t\mapsto (\cos t,\sin t)$ and endow with arc-length measure. For $x,y\in\mathbb T$, we write
\[
d(x,y)\coloneqq \min_{k\in\mathbb Z}|x-y+2k\pi|\in[0,\pi],
\]
for the geodesic distance. If $\Phi$ is $2\pi$-periodic, we keep the shorthand $\Phi(x-y)$ for $\Phi(\tilde x-\tilde y)$, where $\tilde x,\tilde y\in\mathbb R$ are arbitrary representatives of $x,y$.
Under identification $\mathbb T\simeq \mathbb S^1$, the arcs in $\mathbb T$ are exactly the spherical caps of $\mathbb S^1$.

\begin{lemma}[Set rearrangement on the circle]
\label{lem:circle-rearrangement}
Let $\Phi:\mathbb R\to\mathbb R$ be a nonnegative, even, $2\pi$-periodic kernel that is strictly decreasing on $(0,\pi)$. Then among measurable sets $E\subset\mathbb T$ of fixed measure, the functional
\[
E\longmapsto \int_E\int_E \Phi(x-y)\,\dd x\,\dd y
\]
is maximized exactly by arcs, that is, by connected subsets of $\mathbb T$, up to rotation.
\end{lemma}

\begin{proof}
Because $\Phi$ is even and $2\pi$-periodic, there exists a nonnegative function $L:[0,\pi]\to\mathbb R$ such that
\[
\Phi(x-y)=L(d(x,y))
\qquad \text{for} \, x,y\in\mathbb T,
\]
and $L$ is strictly decreasing on $(0,\pi)$. 
We proceed to 
apply \cite[Corollary 7.1]{Baernstein2019} and Fubini's theorem with
\(
f=g=\chi_E.
\)
Then the symmetric decreasing rearrangement (see \cite{Bur}) $E^{\#}$ of $E$ is the spherical cap of the same measure, hence an arc in $\mathbb T$, and
\[
\int_E\int_E \Phi(x-y)\,\dd x\,\dd y
\le
\int_{E^{\#}}\int_{E^{\#}} \Phi(x-y)\,\dd x\,\dd y.
\]
For the equality case, apply \cite[Theorem 7.3(b)]{Baernstein2019} with the same choice of non-constant $f=g=\chi_E$ and strictly supermodular $\Psi(s,t)=st$. 
Since $L$ is strictly decreasing, the equality implies that $\chi_E$ is almost everywhere in agreement with $\chi_{E^{\#}}\circ \varphi$ for some orthogonal map $\varphi$ of $\mathbb S^1$. Therefore, $E$ is an arc up to an isometry of $\mathbb S^1$, and hence up to rotation.
\end{proof}

The next lemma describes the geometry of the bang-bang curvature profiles and is closely related to Kallay’s characterization, \cite{Kallay}
\begin{lemma}[Geometry of bang-bang profiles]
\label{lem:bang-bang-geometry}
Fix an odd integer $n\ge3$ and $x_0\in\R$. Let $K$ be the convex body reconstructed in Lemma~\ref{lem:reconstruct-rho} from
\[
\rho(\theta)=\frac w2\Bigl(1+\operatorname{sgn}(\cos(n\theta-x_0))\Bigr).
\]
Then $K$ is the centered regular Reuleaux $n$-gon of width $w$.
\end{lemma}

\begin{proof}
Set $\ell\coloneqq \frac{n-1}{2}$, and for $j=0,\dots,n-1$ define, modulo $2\pi$,
\[
I_j=\Bigl(\frac{x_0-\pi/2+2j\pi}{n},\frac{x_0+\pi/2+2j\pi}{n}\Bigr),
\qquad
J_j=\Bigl(\frac{x_0+\pi/2+2j\pi}{n},\frac{x_0+3\pi/2+2j\pi}{n}\Bigr).
\]
Then $\rho=w$ on each $I_j$ and $\rho=0$ on each $J_j$, and these intervals alternate around the circle.

By Lemma \ref{lem:reconstruct-rho}, the support function $h$ of $K$ satisfies $h+h''=\rho$ in the sense of distributions. Since $\rho$ is constant on each open interval $I_j$ and $J_j$, and $h''=\rho-h$ in the sense of distributions on each interval, 
by \cite[Proposition 1.3.18(4)]{OW}, this implies that $h$ is smooth on each such interval and the identity holds there pointwise. On each interval define
\(
\gamma(\theta)=h(\theta)u(\theta)+h'(\theta)u^\perp(\theta).
\)
By Proposition \ref{definitiongammaprime}, $\gamma'=(h+h'')u^\perp=\rho\,u^\perp$ on each interval.

If $\theta\in I_j$, then $\rho(\theta)=w$, so
\[
(\gamma(\theta)-w u(\theta))'=\gamma'(\theta)-w u^\perp(\theta)=0.
\]
Hence, for $\theta\in I_j$,
\(
\gamma(\theta)=c_j+w u(\theta)
\)
for some constant vector $c_j$, so the image of $I_j$ is an arc of a circle of radius $w$ centered at $c_j$.
If $\theta\in J_j$, then $\rho(\theta)=0$, hence $\gamma'(\theta)=0$ and $\gamma$ is constant on $J_j$. Denote this constant point by $v_j$. Thus, each interval $J_j$ collapses to a vertex.

Next, constant width gives $h(\theta+\pi)=w-h(\theta)$ and $h'(\theta+\pi)=-h'(\theta)$. Therefore,
\[
\gamma(\theta+\pi)=\gamma(\theta)-w u(\theta).
\]
If $\theta\in I_j$, then using $n=2\ell+1$ we compute
\begin{align*}
I_j+\pi
&=
\Bigl(\frac{x_0-\pi/2+(2j+n)\pi}{n},\frac{x_0+\pi/2+(2j+n)\pi}{n}\Bigr)\\
&=
\Bigl(\frac{x_0+\pi/2+2(j+\ell)\pi}{n},\frac{x_0+3\pi/2+2(j+\ell)\pi}{n}\Bigr)
=J_{j+\ell}
\end{align*}
modulo $2\pi$. In particular, $\rho(\theta+\pi)=0$ on that interval, so $\gamma$ is constant there and therefore $\gamma(\theta+\pi)=v_{j+\ell}$. The previous identity becomes
\[
\gamma(\theta)=v_{j+\ell}+w u(\theta)
\qquad (\theta\in I_j).
\]
So the center of each circular arc is one of the vertices.

To identify the endpoints of the arc corresponding to $I_j$, note that the left endpoint of $I_j$ is the right endpoint of $J_{j-1}$, and the right endpoint of $I_j$ is the left endpoint of $J_j$. Since
\[
h''=\rho-h\in L^\infty(0,2\pi),
\]
in particular, the functions $h$ and $h'$ are continuous on the circle, hence so is
\[
\gamma(\theta)=h(\theta)u(\theta)+h'(\theta)u^\perp(\theta).
\]
Therefore, the one-sided limits of $\gamma$ along $I_j$ at its endpoints agree with the constant values on the adjacent $J$-intervals. Consequently the image of $I_j$ is the circular arc joining $v_{j-1}$ to $v_j$.

Since $I_j$ has length $\pi/n>0$, this arc is nondegenerate, so $v_{j-1}\neq v_j$. Moreover, $C_n$-invariance of $K$ implies
\[
\gamma\Bigl(\theta+\frac{2\pi}{n}\Bigr)=\mathcal R_{2\pi/n}\gamma(\theta).
\]
If $\theta\in J_j$, then $\theta+\frac{2\pi}{n}\in J_{j+1}$ modulo $2\pi$, hence
\[
v_{j+1}=\gamma\Bigl(\theta+\frac{2\pi}{n}\Bigr)=\mathcal  R_{2\pi/n}\gamma(\theta)=\mathcal R_{2\pi/n}v_j.
\]
Thus, the vertices form a single $C_n$-orbit of size $n$, namely the vertex set of a regular $n$-gon centered at the origin.

We have shown that the boundary of $K$ consists of $n$ circular arcs of radius $w$, each joining two consecutive vertices of that regular $n$-gon and centered at another vertex. This is exactly the usual construction of the centered $R_n(w)$.
\end{proof}

\subsection{Rotational  Symmetry \texorpdfstring{$C_n$}{Cn}}

\begin{theorem}[Odd cyclic symmetry]
\label{thm:cyclic-odd}
Among all planar convex bodies of constant width $w$ with $C_n$ rotational symmetry, where $n\ge3$ is odd, the minimizers are precisely the centered $R_n(w)$. Equivalently, the minimizer is unique up to rotation.
\end{theorem}

\begin{proof}
By Corollary \ref{cor:no-relaxation}, it suffices to classify the maximizers of $\mathcal J$ on $\mathcal F_w^{C_n}$. Let $f\in\mathcal F_w^{C_n}$ be an arbitrary maximizer. By Proposition \ref{prop:bang-bang-maximizer}, every maximizer is bang-bang, so
\[
f(\theta)=\frac w2\,\sigma(\theta),\qquad \sigma(\theta)\in\{-1,1\}\ \text{a.e.}
\]
and the symmetry relations
\(
f\Bigl(\theta+\frac{2\pi}{n}\Bigr)=f(\theta),
f(\theta+\pi)=-f(\theta)
\)
still hold.
Rescale to the variable $x=n\theta$ and define
\[
g(x)\coloneqq \frac{2}{w}\,f\Bigl(\frac{x}{n}\Bigr),\qquad x\in[0,2\pi].
\]
Then $g$ is $2\pi$-periodic, takes only the values $\pm1$ almost everywhere, and satisfies
\[
g(x+\pi)=-g(x)\qquad\text{a.e.}
\]
Indeed, for $n=2\ell+1$, by the symmetry relations,
\[
g(x+\pi)=\frac{2}{w}f\Bigl(\frac{x}{n}+\frac{\pi}{n}\Bigr)=\frac{2}{w}f\Bigl(\frac{x}{n}+\pi-\frac{2\ell\pi}{n}\Bigr)=-\frac{2}{w}f\Bigl(\frac{x}{n}\Bigr)=-g(x).
\]
Let
\(
E\coloneqq\{x\in[0,2\pi):g(x)=1\}.
\)
Since $g(x+\pi)=-g(x)$, the sets $E$ and $E+\pi$ are complementary modulo $2\pi$, hence for the Lebesgue measure
\(
|E|=\pi,
g=2\chi_E-1.
\)

Recall the rescaled kernel $H_n$ introduced in equation \eqref{HnKernel}, equivalently, it is the kernel from Proposition \ref{prop:objective-kernel} after the change of variables $x=n\theta$.
If
\[
g(x)=\sum_{\substack{m\ge1\\m\text{ odd}}}(A_m\cos(mx)+B_m\sin(mx)),
\]
then $A_m=\frac{2}{w}\alpha_{mn}$ and $B_m=\frac{2}{w}\beta_{mn}$. Therefore,
\[
\mathcal J(f)=\frac{\pi w^2}{8}\sum_{\substack{m\ge1\\m\text{ odd}}}\frac{A_m^2+B_m^2}{n^2m^2-1}.
\]
Applying the same convolution-Parseval argument as in the proof of Proposition \ref{prop:objective-kernel} gives
\[
\mathcal J(f)=\frac{w^2}{8\pi}\int_0^{2\pi}\int_0^{2\pi} g(x)H_n(x-y)g(y)\,\dd x\,\dd y.
\]
Since $H_n$ has no constant Fourier mode,
\(
\int_0^{2\pi}H_n(t)\,\dd t=0.
\)
Using $g=2\chi_E-1$, we expand
\begin{align*}
\int_0^{2\pi}\!\int_0^{2\pi} g(x)H_n(x-y)g(y)\,\dd x\,\dd y&=
4\int_E\!\int_E H_n(x-y)\,\dd x\,\dd y -2\int_E\!\int_0^{2\pi} H_n(x-y)\,\dd y\,\dd x\\
&-2\int_0^{2\pi}\!\int_E H_n(x-y)\,\dd y\,\dd x  +\int_0^{2\pi}\!\int_0^{2\pi} H_n(x-y)\,\dd y\,\dd x.
\end{align*}
For each fixed $x$, the change of variables $t=x-y$ and $2\pi$-periodicity give
\[
\int_0^{2\pi} H_n(x-y)\,\dd y=\int_0^{2\pi}H_n(t)\,\dd t=0,
\]
so the last three terms vanish. Hence
\[
\mathcal J(f)=\frac{w^2}{2\pi}\int_E\int_E H_n(x-y)\,\dd x\,\dd y.
\]
Thus maximizing $\mathcal J$ is equivalent to maximizing the set functional
\(
I(E)\coloneqq \int_E\int_E H_n(x-y)\,\dd x\,\dd y
\)
among measurable subsets $E\subset[0,2\pi)$ satisfying $E+\pi=E^c$ modulo null sets. In particular, every admissible set $E$ has measure $\pi$.
Indeed, Lebesgue measure is translation invariant, so from $E+\pi=E^c$ modulo null sets we get
\(
|E|=|E+\pi|=|E^c|=2\pi-|E|,
\)
hence $|E|=\pi$.

By Lemma \ref{lem:Hn-kernel}, $H_n$ is even and strictly decreasing on $(0,\pi)$.
Choose $C> -\min_{t\in[0,2\pi]}H_n(t)$. Since $|E|=\pi$ is fixed, maximizing $I(E)$ is equivalent to maximizing
\[
I_C(E)\coloneqq \int_E\int_E \bigl(H_n(x-y)+C\bigr)\,\dd x\,\dd y.
\]
Moreover,
\(
I_C(E)=I(E)+C|E|^2=I(E)+C\pi^2,
\)
so $I$ and $I_C$ have exactly the same maximizers in the admissible class. The choice of $C$ ensures that the kernel $H_n+C$ is nonnegative, even, and strictly decreasing on $(0,\pi)$. Therefore, Lemma \ref{lem:circle-rearrangement} applies to the enlarged class of all measurable subsets of the circle with measure $\pi$,
\[
I_C(E)\le I_C(E^*), \quad \text{for the centered arc} \, E^*, \quad |E^*|=\pi,
\]
and equality can occur only when $E$ is itself an arc up to rotation.
Note that $E^*$ is a half-circle, so it also satisfies
\(
E^*+\pi=(E^*)^c
\)
modulo null sets. Hence $E^*$ belongs to the original admissible class, and the maximal value over the enlarged class is already attained within the original one. It follows that every maximizer for the original problem must be an arc of length $\pi$. Consequently there exists $x_0\in[0,2\pi)$ such that, modulo null sets,
\[
E=\bigl(x_0-\tfrac{\pi}{2},\,x_0+\tfrac{\pi}{2}\bigr)
\quad\text{on} \, \, \mathbb{S}^1.
\]
Equivalently,
\(
g(x)=\operatorname{sgn}(\cos(x-x_0))
\)
a.e., and therefore
\[
f(\theta)=\frac w2\operatorname{sgn}(\cos(n\theta-x_0))\quad\text{a.e. for some phase shift} \, x_0.
\]
Passing back to the curvature radius gives
\[
\rho(\theta)=\frac w2\Bigl(1+\operatorname{sgn}(\cos(n\theta-x_0))\Bigr)
\quad\text{a.e.}
\]
By Lemma \ref{lem:bang-bang-geometry}, the reconstructed convex body is the centered regular Reuleaux $n$-gon of width $w$. Changing $x_0$ merely rotates the body, and every rotation arises in this way. Since $f$ was an arbitrary maximizer, every minimizer arises in this manner. Therefore the minimizers are precisely the centered regular Reuleaux $n$-gons, unique up to rotation.
\end{proof}

\subsection{Dihedral Symmetry \texorpdfstring{$D_n$}{Dn}}

All previous analysis was carried out under the assumption of $C_n$ rotational symmetry. We continue to assume that $n$ is odd, but now we impose an additional reflection symmetry. Recall that a body has $D_n$ symmetry if it has $C_n$ symmetry and is invariant under reflection across a prescribed axis. 

\begin{theorem}[Odd dihedral symmetry]
Among the $D_n$-invariant planar convex bodies of constant width $w$, where $n \geq 3$ is odd, the minimizers are exactly the two centered $R_n(w)$ invariant under the prescribed $D_n$-action. Equivalently, they are the two orientations in which the prescribed reflection axis contains a vertex.
\end{theorem}

\begin{proof}
Let $\mathcal K_w^{D_n}\subseteq\mathcal K_w^{C_n}$ denote the
corresponding symmetry classes. By Theorem \ref{thm:cyclic-odd}, every minimizer in
$\mathcal K_w^{C_n}$ has curvature density
\[
\rho_{x_0}(\theta)
=
\frac{w}{2}
\left(1+\operatorname{sgn}\bigl(\cos(n\theta-x_0)\bigr)\right)
\quad\text{a.e.}
\]
for some $x_0\in\mathbb R/2\pi\mathbb Z$, and the corresponding body
is a centered regular Reuleaux $n$-gon.

Reflection across the $x$-axis corresponds to $\theta\mapsto-\theta$.
By the uniqueness in Corollary \ref{cor:no-relaxation}, the corresponding body is
reflection-invariant if and only if
\[
\rho_{x_0}(-\theta)=\rho_{x_0}(\theta)
\quad\text{a.e.}
\]
The set on which $\rho_{x_0}=w$ is the half-circle
\[
E_{x_0}
=
\left(x_0-\frac{\pi}{2},x_0+\frac{\pi}{2}\right)
\quad\text{modulo }2\pi.
\]
Thus $\rho_{x_0}$ is even if and only if
$E_{x_0}=-E_{x_0}$ modulo null sets, which holds precisely when
\[
x_0\equiv0
\quad\text{or}\quad
x_0\equiv\pi
\pmod{2\pi}.
\]

Both of these profiles are reflection-invariant, so the minimum in the
$C_n$-invariant class is attained in the $D_n$-invariant class. Since
$\mathcal K_w^{D_n}\subseteq\mathcal K_w^{C_n}$, the two minimum values
are equal. Conversely, every $D_n$-invariant minimizer must therefore
also be a $C_n$-invariant minimizer, so the preceding argument shows
that no other phases are possible.

By Lemma \ref{lem:bang-bang-geometry}, the two resulting bodies are centered regular Reuleaux
$n$-gons. Moreover,
\[
\rho_{\pi}(\theta)
=
\rho_0\left(\theta-\frac{\pi}{n}\right),
\]
so they differ by a rotation through $\pi/n$. These are exactly the two
orientations in which the prescribed reflection axis contains a vertex.
\end{proof}

\section{Area of the Reuleaux \texorpdfstring{$n$}{n}-gon}
For the sake of completeness, we compute the area of the minimizing shape $R_n(w)$.
\begin{proposition}
For odd $n\ge3$, the regular Reuleaux $n$-gon of width $w$ has area
\[
\operatorname{Area}(R_n(w))
=
\frac{w^2}{2}\left(\pi-n\tan\frac{\pi}{2n}\right).
\]
\end{proposition}

\begin{proof}
Let $V_0,\dots,V_{n-1}$ be the vertices of $R_n(w)$, and let $P_n$ denote its convex hull. Let
\(
\alpha\coloneqq \frac{\pi}{2n}.
\)
Since each boundary arc of $R_n(w)$ has radius $w$ and subtends angle $\pi/n=2\alpha$, the circumradius $R$ of $P_n$ is determined by
$w=2R\cos\alpha$,
hence
$R=\frac{w}{2\cos\alpha}$.

Body $R_n(w)$ is obtained from $P_n$ by attaching to each side a circular segment $S_n$ of radius $w$ and angle $\pi/n$. 
Therefore, 
\[
\operatorname{Area}(R_n(w))
=\operatorname{Area}(P_n)+n\,\operatorname{Area}(S_n).
\]
Since $P_n$ is a regular $n$-gon with circumradius $R$,
\(
\operatorname{Area}(P_n)=\frac{n}{2}R^2\sin\frac{2\pi}{n}
=\frac{n}{2}\cdot \frac{w^2}{4\cos^2\alpha}\,\sin(4\alpha).
\)
Then segment $S_n$ is a sector of radius $w$ and angle $\pi/n=2\alpha$, minus the isosceles triangle cut off by the corresponding chord. Hence
\(
\operatorname{Area}(S_n)
=\frac{w^2}{2}\cdot \frac{\pi}{n}-\frac{w^2}{2}\sin\frac{\pi}{n}.
\)
Therefore
\begin{align*}
\operatorname{Area}(R_n(w))
=\ \frac{n w^2}{8\cos^2\alpha}\sin(4\alpha)+\frac{\pi w^2}{2}-\frac{n w^2}{2}\sin(2\alpha).
\end{align*}
By $\sin(4\alpha)=4\sin\alpha\cos\alpha\cos(2\alpha)$ and $\sin(2\alpha)=2\sin\alpha\cos\alpha$, we obtain
\begin{align*}
\operatorname{Area}(R_n(w))
= \frac{\pi w^2}{2}+\frac{n w^2\sin\alpha}{2\cos\alpha}\bigl(\cos(2\alpha)-2\cos^2\alpha\bigr).
\end{align*}
The application of the the identity $\cos(2\alpha)=2\cos^2\alpha-1$ concludes the proof.
\end{proof}

As expected, for $n=3$, this gives
\(
\operatorname{Area}(R_3(w))
=\frac{w^2}{2}(\pi-\sqrt3).
\)

\subsection*{Acknowledgment}
This work was partially carried out at University of the Fraser Valley during a research visit by Javier Falcó, who gratefully acknowledges the support and hospitality of the Department of Mathematics and Statistics at UFV. The authors acknowledge the use of artificial intelligence–based tools for language editing and stylistic improvements
of the manuscript. The authors remain fully responsible for the content of this work.

\end{document}